\documentclass[a4paper,11pt]{article}
\usepackage{latexsym}
\usepackage{amsfonts, amsthm}
\usepackage[all,cmtip]{xy}
\usepackage{rotating}
\usepackage{amsmath}
\usepackage{bbm}
\usepackage{amssymb}
\usepackage{scalefnt}
\usepackage{amscd}
\usepackage[english]{babel}
\usepackage{makeidx}
\usepackage[latin1]{inputenc}
\usepackage{fancyhdr}
\usepackage{diagrams}
\usepackage[font={small,it},aboveskip=0pt]{caption}
\usepackage{float}
\newtheorem{teo}{Theorem}[section]
\newtheorem{formula}{Formula}
\newtheorem{prop}[teo]{Proposition}
\newtheorem*{prop*}{Proposition} 
\newtheorem{lemma}[teo]{Lemma} 
\newtheorem*{teo*}{Theorem}

\newcommand{\rmk}{\textit{Remark. }}

\newtheorem{defi}[teo]{Definition}
\newtheorem*{defi*}{Definition}

\begin{document}

\title{The Brauer-Kuroda formula for higher $S$-class numbers in dihedral extensions of number fields}
\author{Luca Caputo\footnote{Supported by an IRCSET fellowship.}}
\date{\today}
\maketitle

\noindent\textsc{Abstract -}\begin{small}{Let $p$ be an odd prime and let $L/k$ be a Galois extension of number fields whose Galois group is isomorphic to the dihedral group of order $2p$. Let $S$ be a finite set of primes of $L$ which is stable under the action of $\mathrm{Gal}(L/k)$. The Lichtenbaum conjecture on special values of the Dedekind zeta function at negative integers, together with Brauer formalism for Artin's $L$-functions, gives a (conjectural) formula relating orders of motivic cohomology groups of rings of $S$-integers and higher regulators of the subextensions of $L/k$. In analogy with the classical case of special values at $0$, we give an algebraic proof of this formula, i.e. without using the Lichtenbaum conjecture nor Brauer formalism. Our method also gives an interpretation of the regulator term as a higher unit index.\\

\noindent 2000 Mathematical Subject Classification: Primary $19F27$ Secondary $11R20$}\end{small}

\section{General setting and statement of the main result}\label{intro}
For a number field $E$ and a finite set of places $S$ of $E$, denote by $\zeta^S_{E}$ the Dedekind $S$-zeta function of $E$. For a complex number $s$, denote by $\zeta^S_{E}(s)^*$ the special value of $\zeta^S_{E}$ at $s$ (i.e. the first nontrivial coefficient of the Laurent expansion of $\zeta^S_{E}$ around $s$). Then by a well-known result of Dirichlet we have the formula 
\begin{equation}\label{sp0}
\zeta^S_{E}(0)^*=-\frac{h^S_E}{w_E}R^S_E
\end{equation}
where $h^S_E$ is the class number of the ring $\mathcal{O}^S_E$ of $S$-integers of $E$, $w_E$ is the order of the group of roots of unity of $E$ and $R^S_E$ is the regulator of $(\mathcal{O}^S_E)^\times$. There are conjectural analogues of this formula when $0$ is replaced by negative integers: more precisely, for any integer $m\geq 2$, a $S$-version of the Lichtenbaum conjecture reads (see \cite{Li}, \cite{Ko}, \cite{Ka}...)
\begin{equation}\label{spi}
\zeta^S_E(1-m)^*=(-1)^{t_{E,\,m}}\frac{h^S_{E,\,m}}{w_{E,\,m}}R_{E,\,m}
\end{equation}
Here $h^S_{E,\,m}$ is the order of the motivic cohomology group $H^2(\mathcal{O}^S_E,\,\mathbb{Z}(m))$ (which is finite  by the Bloch-Kato "conjecture"), $w_{E,\,m}$ is the order of the torsion subgroup of the motivic cohomology group $H^1(\mathcal{O}^S_E,\,\mathbb{Z}(m))$ (which is a finitely generated $\mathbb{Z}$-module by the Bloch-Kato "conjecture") and $R_{E,\,m}$ is the (motivic) regulator of $H^1(\mathcal{O}^S_E,\,\mathbb{Z}(m))$ (see Section \ref{regulcomp}).
In this paper we use the definition of motivic cohomology in terms of Bloch's higher Chow groups, in other words 
$$H^j(\mathcal{O}^S_E,\,\mathbb{Z}(m)):=CH^m(\mathrm{Spec}(\mathcal{O}^S_E),\,2m-j)$$
Finally, $t_{E,\,m}\in \mathbb{N}$ is given by 
$$t_{E,\,m}=\left\{\begin{array}{cl}1&\textrm{if $m\equiv1$ mod $4$}\\r_1(E)+r_2(E)&\textrm{if $m\equiv2$ mod $4$}\\r_1(E)&\textrm{if $m\equiv3$ mod $4$}\\r_2(E)&\textrm{if $m\equiv0$ mod $4$}\end{array}\right.$$ 
where $r_1(E)$ (resp. $r_2(E)$) is the number of real places (resp. complex places) of $E$.\\

We now list some known facts about the Lichtenbaum conjecture and motivic cohomology: for any further detail and reference to relevant literature we refer the reader to Kolster's and Kahn's excellent surveys (\cite{Ko} and \cite{Ka}). Here we just recall that, if $m\geq 2$, we have
\begin{equation}\label{mok}
H^1(\mathcal{O}_E^S,\,\mathbb{Z}(m))\otimes \mathbb{Q}\cong K_{2m-1}(\mathcal{O}_E^S)\otimes \mathbb{Q}
\end{equation}
and, since the Bloch-Kato conjecture on the Galois symbol holds (see \cite{We}), we also have  
\begin{equation}\label{moet}
H^j(\mathcal{O}^S_E,\,\mathbb{Z}(m))\otimes_{\mathbb{Z}}\mathbb{Z}_\ell\cong H_{\acute{e}t}^{j}(\mathcal{O}^S_E\textstyle{[\frac{1}{\ell}]},\,\mathbb{Z}_\ell(m)) \quad j=1,\,2
\end{equation} 
where $\ell$ is any prime. Whenever a Galois action is defined on $E$, the above isomorphisms are invariant under this action.\\

Then thanks to Wiles' proof of the main conjecture in Iwasawa theory, the Lichtenbaum conjecture is known for $m\geq 2$ even and $E$ totally real abelian and it is also known to hold up to power of $2$ for $m\geq 2$ even and $E$ totally real. More generally it is known to hold up to power of $2$ for $E$ abelian and any $m\geq 2$.\\

Note that $H^1(\mathcal{O}^S_E,\,\mathbb{Z}(m))$ does not depend on $S$ (use for example (\ref{moet}) and the correspondig property for étale cohomology groups, see Lemma \ref{trivialrelations}): so we shall omit the reference to $S$ in $w_{E,\,m}^S$ and $R_{E,\,m}^S$.\\

Now let $p$ be an \emph{odd} prime. Let $D=D_p$ denote the dihedral group of order $2p$: in particular 
$$D=\langle\tau,\,\sigma\,|\,\tau^{p}=\sigma^2=1,\,\sigma\tau\sigma=\tau^{-1}\rangle$$ 
Let $L/k$ be a Galois extension of number fields such that $\mathrm{Gal}(L/k)\cong D$ (in the rest of this paper we shall identify those groups). Let $K$ (resp. $K'$) be the subfield of $L$ fixed by $\langle\sigma\rangle$ (resp. by $\langle\tau^2\sigma\rangle$): in particular $K'=\tau(K)$. Let $F$ be the subfield of $L$ fixed by $\langle\tau\rangle$: set $G=\mathrm{Gal}(L/F)$ and $\Delta=\mathrm{Gal}(F/k)$. \\

Let $S$ be a finite set of places of $L$ which is stable under the action of $D$ and contains the archimedean primes (we shall consider only sets of primes containing the archimedean ones, so we will not further mention this property). For any subfield $E$ of $L$ containing $k$, the set of places of $E$ which lie below those of $S$ will be denoted by $S_E$ or simply again by $S$ if no misunderstanding is possible. The existence of the nontrivial $D$-relation (in the sense of \cite{DD}, see Definition 2.1 and Example 2.4 of that paper) 
\begin{equation}\label{relation}
\{1\} - 2\langle \sigma\rangle - G +2 D
\end{equation}
together with Brauer formalism for Artin's $L$-functions, gives the following formula
\begin{equation}\label{zetas}
\zeta^S_{L}(s)=\zeta^S_{F}(s)\frac{\zeta^S_{K}(s)^2}{\zeta^S_k(s)^2}
\end{equation}
Considering the special value at $0$ of (\ref{zetas}) and using (\ref{sp0}), we get 
\begin{equation}\label{formula}
h^S_{L}=h^S_{F}\frac{(h^S_{K})^2}{(h^S_k)^2}\cdot\frac{w_k^2w_{L}}{w_{K}^2w_{F}}\cdot\frac{(R^S_{K})^2R^S_{F}}{(R^S_k)^2R^S_{L}}
\end{equation}
which is commonly referred to as the (classical) Brauer-Kuroda formula (for dihedral extensions of order $2p$). It can be shown that the $w$-factor is actually trivial. More interestingly, the factor involving regulators and roots of unity can be expressed as an index of subgroups whose definition involves units of subextensions of $L/k$ (see \cite{Ba}, \cite{Ja}, \cite{HK}, \cite{Lem}, ...). Considering special values at negative integers and using (\ref{spi}), we get of course, for any $m\geq2$, a conjectural analogue of (\ref{formula}) 
\begin{equation}\label{bk}
h^S_{L,\,m}= h^S_{F,\,m}\frac{(h^S_{K,\,m})^2}{(h_{k,\,m})^2}\cdot\frac{(w^S_{k,\,m})^2w^S_{L,\,m}}{(w^S_{K,\,m})^2w^S_{F,\,m}}\cdot\frac{(R^S_{K,\,m})^2R^S_{F,\,m}}{(R^S_{k,\,m})^2R^S_{L,\,m}}
\end{equation}
(it is easy to see that indeed the signs appearing in (\ref{spi}) cancel each other out in (\ref{bk}), use for example Lemma 2.37 of \cite{DD}).\\

In this paper we prove (\ref{bk}) without using the Lichtenbaum conjecture and actually in an algebraic way, i.e. we make no use of $L$-functions at all. It is worth noting that the $w$-term in the above formula is trivial, as in the classical case, thanks to the following lemma (if $A$ is a ring and $M$ is an $A$-module, $\mathrm{tor}_{A}(M)$ denotes the torsion submodule of $M$). 

\begin{lemma}\label{torsion}
Let $S$ be a finite set of places of $L$ which is invariant under the action of $\mathrm{Gal}(L/k)$. Then
$$w_{L,\,m}=w_{F,\,m}\quad \textrm{and}\quad w_{K,\,m}=w_{k,\,m}$$
\end{lemma}
\begin{proof}
Recall that for any number field $E$ and any prime $\ell$, we have ($m\geq 2$)
$$\mathrm{tor}_{\mathbb{Z}}(H^{1}(\mathcal{O}_E^S,\,\mathbb{Z}(m)))\otimes_{\mathbb{Z}}\mathbb{Z}_\ell\cong\mathrm{tor}_{\mathbb{Z}_\ell}(H_{\acute{e}t}^{1}(\mathcal{O}_E^S\textstyle{[\frac{1}{\ell}]},\,\mathbb{Z}_\ell(m)))\cong H^0(E,\,\mathbb{Q}_\ell/\mathbb{Z}_\ell(m))$$
and the latter has cardinality $v_\ell(\kappa_E(\gamma_E)^{m}-1)$, where $v_\ell$ is the $\ell$-adic valuation such that $v_\ell(\ell)=1$, $\kappa_E:\Gamma_E\to \mathbb{Z}_\ell^\times$ is the cyclotomic character evaluated on $\Gamma_E=\mathrm{Gal}(E(\mu_{\ell^\infty})/E)$ and $\gamma_E$ is any topological generator of $\Gamma_E$. Now, since $L/k$ is not abelian, $L\cap F(\mu_{\ell^\infty})=F$. This shows that 
$$v_\ell(\kappa_L(\gamma_L)^{m}-1)=v_\ell(\kappa_F(\gamma_F)^{m}-1)$$
since restriction maps $\Gamma_L$ isomorphically onto $\Gamma_F$. A similar argument apply for $K$ and $k$.\\
\end{proof}

The proof of the conjectural formula above is achieved by summing up the following two results (which are proved respectively in Section \ref{secp} and Section \ref{regulcomp}), which are maybe interesting in their own right. 

\begin{formula}\label{alg}
The following formula holds
$$h^S_{L,\,m}= p^{-\alpha_m}h^S_{F,\,m}\left(\frac{h^S_{K,\,m}}{h^S_{k,\,m}}\right)^2 u_m,$$
where $\alpha_m=\mathrm{rk}_{\mathbb{Z}_p}H_{F,\,m}-\mathrm{rk}_{\mathbb{Z}_p}H_{k,\,m}$ and $$u_m=\frac{(H_{L,\,m}:\,H_{F,\,m}H_{K,\,m}H_{K',\,m})((\overline{H}_{F,\,m})^{\Delta}:\overline{H}_{k,\,m})}{((\overline{H}_{F,\,m})^{D}:\overline{H}_{k,\,m})}$$
where, for a number field $E$, we have set $H_{E,\,m}=H^1(\mathcal{O}^S_E,\,\mathbb{Z}_p(m))$ and  $\overline{H}_{E,\,m}=H_{E,\,m}/\mathrm{tor}_{\mathbb{Z}}H_{E,\,m}$. 
\end{formula}

\begin{formula}\label{ind} 
With notation as in the previous formula, the following formula holds
  $$p^{-\alpha_m}u_m=\frac{(R_{K,\,m})^2R_{F,\,m}}{(R_{k,\,m})^2R_{L,\,m}}$$
\end{formula}

Note that indeed $u_m$ and $\alpha_m$ do not depend on $S$ because $H_{E,\,m}$ doesn't. It will turn out that, as in the classical case, $u_m$ is a power of $p$.\\ 

We will divide the proof of Formula \ref{alg} in two parts, studying separatedly $p$-parts and $\ell$-parts for any prime $\ell\ne p$ (using then (\ref{moet}) to glue all parts together). It should be stressed that the proof for $\ell$-parts with $\ell\ne p$ is really much easier: using just the fact that the cohomology groups involved are cohomological Mackey functors (in the sense of Dress, see for example \cite{Bo}) and that $D=D_p$ is not $\ell$-hypoelementary, we get even more precise structural relations. The proof for $p$-parts uses mainly descent and co-descent results for étale cohomology groups (which are described in \cite{KM}, see also \cite{Ko}). This is essentially the only arithmetic information which is needed, the rest of the proof being a technical algebraic computation. The proof of Formula \ref{alg} given here can probably be generalized without too much effort to metabelian groups whit commutators subgroup of order a power of $p$ and index coprime with $p$. However that  seems not to be the best approach for the general case of an arbitrary Galois group. We believe anyway that there exists a general algebraic proof of Brauer-Kuroda formulas for an arbitrary finite Galois extension (a higher analogue of \cite{dS}). It is worth noting that any proof of the classical version of Formula \ref{alg}  (see \cite{Ja}, \cite{HK}, \cite{Lem}, ...) uses class field theory and genus theory and considers $p$-parts and $\ell$-parts for $\ell\ne p$ separatedly.\\ 

In the last section we perform the proof of Formula \ref{ind}.  The translation of $u_m$ in terms of higher (or motivic) regulators is done using methods from representation theory which have been introduced by the Dokchitser brothers (see for example \cite{DD}). In particular we follow the strategy of Bartel (see \cite{Ba}), who used Dokchitser's ideas to prove a statement analogous to ours in the classical case. In order to use these techniques, we need a higher version of Dirichlet's theorem on the Galois structure of units, which we state and prove, since we could not find it in the literature.\\  

%It is maybe worth noting that Brauer-Kuroda formulas do not shed much light on the Lichtenbaum conjecture. Of course, there may well exist a general proof of Brauer-Kuroda formulas for arbitrary finite Galois extensions which make no use of the Lichtenbaum conjecture (see \cite{dS} for the classical case or \cite{DD}, footnote at page 7, for the case of elliptic curves). But even if all the possibile Brauer-Kuroda formulas hold, this only says that a possible error term in the Lichtenbaum conjecture is trivial on relations (again in the sense of \cite{DD}). %Finally one could wonder whether known cases of the Lichtenbaum conjecture together with the appropriate Brauer-Kuroda formula could give new cases of the conjecture. For example, a good try would be a Galois extension $E/k$ of a totally real number field $k$ with $\mathrm{Gal}(E/k)$ isomorphic to the quaternion group $Q$, since all the subgroups of $Q$ are normal with abelian quotient (and it is easy to see that any proper subextension of $E/k$ has to be totally real). Unfortunately there is no $Q$-relation which includes the regular representation and involves exclusively representations induced from trivial representations of subgroups. But this is precisely the condition needed to apply Brauer formalism of $L$-functions and gather results on $E$ from information coming from its proper subextensions. Thus, Lichtenbaum formulas for proper subextensions of $E/k$ do not imply (at least in this way) the Lichtenbaum conjecture for $E$.

\emph{Notation and standard results} 
\begin{itemize}
	\item As before, throughout the paper, if $A$ is a commutative ring and $M$ is an $A$-module, $\mathrm{tor}_{A}(M)$ denotes the torsion submodule of $M$. We will also use the notation $\overline{M}$ for $M/\mathrm{tor}_A M$ without any specific mention to $A$, since it will be clear from the context which is the ring we are considering. Finally, for any $a\in A$, we set $M[a]=\{m\in M\,|\, am=0\}$.
	\item Let $H$ be a finite group. We denote by $N_{H}=\sum_{h\in H}h\in A[H]$ the norm element and by $I_H\subseteq A[H]$ the augmentation ideal. If $B$ is a $A[H]$-module, we use the following notation 	
\begin{itemize}
  \item $B^H=\{b\in B\, \mid\, hb=b\,\, \textrm{for all $h\in H$}\}$;
  \item $B_H=B/I_HB$;
	\item $B[N_H]=\{b\in B\, \mid\, N_Hb=0\}$.
\end{itemize}
If $\ell$ is a prime, $A=\mathbb{Z}_\ell$ and $H$ is a $q$-group for some prime $q\ne \ell$, then 
$$B^H=N_HB \quad\textrm{and}\quad B[N_H]=I_HB$$ 
since $B$ is $q$-divisible (being a $\mathbb{Z}_\ell$-module) and hence $H$-cohomologically trivial. 
\end{itemize}

\emph{Aknowledgements} I would like to thank Kevin Hutchinson for many enlightening discussions on this subject and Manfred Kolster for making me aware of the work of the Dokchitser brothers.

\section{A formula relating higher class numbers and a higher units index}\label{secp}
In this section we prove Formula \ref{alg}. First we study the $p$-part of the problem which is the most delicate.\\ The natural number $m\geq 2$ will be fixed throughout the section. For any number field $E$ such that $k\subseteq E\subseteq L$ and any finite set $S$ of primes of $L$, we set 
$$U_{E,m}=H_{\acute{e}t}^{1}(\mathcal{O}_E^S\textstyle{[\frac{1}{p}]},\,\mathbb{Z}_p(m))$$
$$A^S_{E,m}=H_{\acute{e}t}^{2}(\mathcal{O}_E^S\textstyle{[\frac{1}{p}]},\,\mathbb{Z}_p(m))$$
(in fact, $U_{E,\,m}$ does not depend on $S$, see Lemma \ref{trivialrelations}) We also fix for this section a finite set $T$ of primes of $L$ such that
\begin{itemize}
	\item $T$ is stable under the action of $D=\mathrm{Gal}(L/k)$;
	\item $T$ contains those primes which ramify in $L/k$;
	\item $T$ contains all the primes above $p$.
\end{itemize}
Since $T$ and $m$ are fixed for this section we will also use the notation $U_E$ for $U_{E,m}$ and $A_E$ for $A^T_{E,m}$. Both $A_E$ and $U_E$ are abelian groups: however, because of their analogies with the ideal class group and the unit group of $E$ respectively, we are going to use multiplicative notation for them.

Note that if $Q$ is a group of automorphisms of $E$ of order $2$, then  $U_E^Q=U_{E^Q}$ and $A_E^Q=A_{E^Q}$: this follows from the fact that $\mathbb{Z}_p(m)$ is $Q$-cohomologically trivial. The following well-known result gives us the description of $G$-descent for $U_E$ and $A_E$ (recall that $G=\mathrm{Gal}(L/F)$). 

\begin{prop}\label{descent} 
The natural map $U_F\to U_L^{G}$ is an isomorphism and we have an exact sequence
$$0\to H^1(G,U_L)\to A_F\to A_L^G\to H^2(G,\,U_L)\to0$$ 
\end{prop}
\begin{proof}
See \cite{KM}, Theorem 1.2. Note that the hypotheses are satisfied thanks the properties we required on $T$.
\end{proof}

We will often identify $U_F$ with its image in $U_L$ (the same will be done with $U_K$, $U_K'$ and $U_k$). In the same way, we will identify $A_K$ and $A_{K'}$ with their images in $A_L$. We record now the following easy lemma which will be used repeatedly for the rest of this section.

\begin{lemma}\label{nuccio} 
Let $M$ be a $2$-divisible $D$-module: then the Tate isomorphisms $\widehat{H}^j(G,M)\cong \widehat{H}^{j+2}(G,M)$ are $\Delta$-antiequivariant, so that in particular, if $H^{j}(G,\,M)$ is finite for any $j\in \mathbb{Z}$, we have
$$
|\widehat{H}^{j}(D,M)|=|\widehat{H}^{j+2}(G,M)|/|\widehat{H}^{j+2}(D,M)|
$$
\end{lemma}
\begin{proof}
We only need to show that Tate's isomorphism is $\Delta$-antiequivariant. Recall that the Tate isomorphism is given by the cup product with a fixed generator $\chi$ of $H^2(G,\mathbb{Z})$:
\begin{equation*}\begin{array}{ccc}
\widehat{H}^i(G,M)&\longrightarrow&\widehat{H}^{i+2}(G,M)\\
x&\longmapsto&x\cup \chi
\end{array}\end{equation*}
The action of $\delta \in \Delta$ on $\widehat{H}^i(G,M)$ is $\delta_*$ in the notation of \cite{NSW}, I.5 and this action is $-1$ on $H^2(G,\mathbb{Z})$ as can immediately be seen through the isomorphism $H^2(G,\mathbb{Z})\cong H^1(G,\,\mathbb{Q}/\mathbb{Z})=\mathrm{Hom}(G,\mathbb{Q}/\mathbb{Z})$ (which comes from the exact sequence $0\to \mathbb{Z}\to \mathbb{Q}\to \mathbb{Q}/\mathbb{Z}\to 0$ and the fact that $\mathbb{Q}$ is $G$-cohomologically trivial being $p$-divisible). Then, by Proposition 1.5.3 of \cite{NSW}, $\delta_{*}(x\cup\chi)=-(\delta_*x)\cup \chi$ which gives the result.
\end{proof}

The next lemma deals with the subgroup $A_KA_{K'}\subseteq A_L$ but there is an analogous version for $U_KU_{K'}\subseteq U_L$ (just replace $A$ by $U$ in the statement).

\begin{lemma}\label{esto}
The subgroup $A_KA_{K'}\subseteq A_L$ is a $D$-module and 
$$A_KA_{K'}=\prod_{j=0}^{p-2}A_K^\tau=\prod_{j=0}^{p-2}A_{\tau(K)}$$ 
Moreover $I_GA_L\subseteq A_KA_{K'}$. 
\end{lemma}
\begin{proof}
For the first assertion, see \cite{HK}, Lemma 1. For the last one, see \cite{Lem}, Lemma 3.3.
\end{proof}

We now start with the proof of the $p$-part of Formula \ref{alg}. 

\begin{lemma}\label{es}
Define $\iota: A_K\oplus A_{K'}\to A_L$ as $\iota(a,\,a')=aa'$. Then there is an exact sequence as follows
$$0\to H^{0}(D,\,A_L)\to A_K\oplus A_{K'}\stackrel{\iota}{\rightarrow} A_L\to H_0(G,\,A_L)/H_0(G,\,A_L)^{\Delta}\to 0$$
\end{lemma}
\begin{proof}
It is easy to see that the map $\mathrm{Ker}\,\iota\to H^{0}(D,\,A_L)$ given by $(a,\,a')\mapsto a$ is indeed an isomorphism.\\ 
As for the cokernel of $\iota$, note that $I_GA_L\subseteq A_KA_{K'}$ by Lemma \ref{esto}. Now the claim follows since 
$$A_{K'}I_GA_L/I_GA_L=(A_L/I_GA_L)^{\langle\tau^2\sigma\rangle}=(A_L/I_GA_L)^{\langle\sigma\rangle}=A_KI_GA_L/I_GA_L$$
Therefore
$$H_0(G,\,A_L)^{\Delta}=(A_L/I_GA_L)^{\Delta}=A_KA_{K'}I_GA_L/I_GA_L=A_KA_{K'}/I_GA_L$$
\end{proof}

\begin{lemma}
The following equality holds
$$|H^0(D,\,A_L)|=\frac{|H^2(D,\,U_L)|\cdot|A_k|}{|H^1(D,\,U_L)|}$$
\end{lemma}
\begin{proof}
Take $\Delta$-invariants of the exact sequence in Proposition \ref{descent}: the sequence stays exact. Then use that $(A_L^G)^{\Delta}=A_L^{D}$, $A_F^\Delta\cong A_k$ and $H^j(G,U_L)^{\Delta}=H^j(D,\,U_L)$ ($j=1,\,2$) because $U_L^G=U_F$ is $\Delta$-cohomologically trivial (being $2$-divisible).
\end{proof} 

The next lemma describe codescent for $A_L$.

\begin{lemma}
The corestriction map induces isomorphisms $H_0(G,\,A_L)\cong A_F$ and $H_0(G,\,A_L)^{\Delta} \cong A_k$.
\end{lemma}
\begin{proof}
For the first isomorphism use \cite{KM}, Proposition 1.3. Then note that the second isomorphism follows from the first one being $\Delta$-equivariant (since corestriction commutes with conjugation).
\end{proof} 

In what follows we shall rewrite the orders of $H^1(D,\,U_L)$ and $H^2(D,\,U_L)$ in terms of certain unit indexes. We first quote a simple lemma which has been used already by Lemmermeyer (see \cite{Lem}, Section 5).

\begin{lemma}\label{lem}
Let $f:B\to B'$ be a homomorphism of abelian groups and let $C$ be a subgroup of finite index in $B$. Then
$$(B:C)=(f(B):f(C))\cdot(\mathrm{Ker} f:\mathrm{Ker} f\cap C)$$
\end{lemma}
\begin{proof}
This is clear because we have the exact sequence
$$0\to (C+\mathrm{Ker}f)/C\to B/C\to f(B)/f(C)\to 0$$
\end{proof}

\rmk The preceding lemma implies in particular the following equality ($B=B'=U_L$, $f=N_G$)
$$(U_L:U_KU_{K'}U_F)=(N_GU_L:N_G(U_KU_{K'}U_F))\cdot(U_L[N_G]:U_L[N_G]\cap U_KU_{K'}U_F)$$
Note that $U_KU_{K'}U_F$ is of finite index in $U_L$ because, for example, $N_G(U_KU_{K'}U_F)$ is of finite index in $N_GU_L$ (since both are of finite index in $U_F=U_L^G$) and $U_L[N_G]\cap U_KU_{K'}U_F$, which contains $I_GU_L$ by Lemma \ref{esto}, is of finite index in $U_L[N_G]$.\\

Recall (see Section \ref{intro}) that $M[p]$ is the submodule of the $\mathbb{Z}_p$-module $M$ which is killed by $p$.

\begin{lemma}\label{qhd}
We have 
$$|H^1(D,\,U_L)|=|U_L[N_G]/\left(U_L[N_G]\cap U_KU_{K'}U_F\right)|\cdot |I_GU_L\cdot U_F[p]/I_GU_L\cdot U_k[p]|$$ 
and 
$$|H^2(D,\,U_L)|=\frac{|U_F/U_F^p|}{|U_k/U_k^p|\cdot|N_GU_L/N_G(U_KU_{K'}U_F)|}$$ 
\end{lemma}
\begin{proof} 
We prove the first assertion. The norm map 
$$U_L\stackrel{1+\sigma}{\longrightarrow} N_{\langle\sigma\rangle} U_L=U_L^{\langle\sigma\rangle}=U_K$$ 
gives a map $U_{L}[N_D]\to U_L[N_G]\cap U_KU_{K'}U_F$. We consider the induced map $$\overline{N}:U_{L}[N_D]\to \left(U_L[N_G]\cap U_KU_{K'}U_F\right)/I_GU_L\cdot U_F[p]$$ 
Note that indeed $U_F[p]\subseteq U_L[N_G]$ since $N_G$ is raising to the $p$-th power on $U_F$ and $I_GU_L\subseteq U_KU_{K'}$ (see Lemma \ref{esto}).
Then we claim that the sequence 
$$0\to \mathrm{Ker}\overline{N}/I_GU_L \to U_L[N_D]/I_GU_L\stackrel{\overline{N}}{\rightarrow} \left(U_L[N_G]\cap U_KU_{K'}U_F\right)/I_GU_L\cdot U_F[p]\to0$$
is exact. Note that indeed $I_GU_L\subseteq \mathrm{Ker}\overline{N}$ since for any $u\in U_L$
$$u^{(1-\tau)(1+\sigma)}=u^{(1-\tau)}\cdot u^{\sigma(1-\tau^{-1})} \in I_GU_L$$
The only nontrivial thing to prove is the surjectivity of $\overline{N}$: take $u\in U_K$, $u'\in U_{K'}$ and $v\in U_F$ such that $N_G(uu'v)=1$. We can find $t,\,t'\in U_L$ such that $t^{1+\sigma}=u$ and $(t')^{1+\tau^2\sigma}=u'$. Then 
\begin{equation}\label{condi}
1=N_G(uu'v)=N_{D}(tt')N_G(v)=N_{D}(tt')v^p
\end{equation}
In particular $v^p\in U_L^D=U_k$. Note that 
\begin{equation}\label{ppi}
U_k\cap U_F^{p}=U_k^p
\end{equation} 
(the surjective map $U_F\stackrel{p}{\rightarrow} U_F^p$ stays surjective after taking $\Delta$-invariants). Hence there exists $w\in U_k$ such that $v^p=w^p$, which implies $v=wv_0$ for some $v_0\in U_F[p]$. Therefore 
$$\overline{N}(tt'w^{\frac{1}{2}})=uu'v\,\,\,\,\mathrm{mod}\,\,\,I_GU_L\cdot U_F[p]$$ 
and the fact that $tt'w^{\frac{1}{2}}\in U_{L}[N_D]$ is exactly (\ref{condi}). This proves that the above short sequence is exact. Now note that $I_DU_L\subseteq \mathrm{Ker}\overline{N}$ (since $I_GU_L\subseteq \mathrm{Ker}\overline{N}$ and $(1+\sigma)(1-\sigma)=0$) and we have an exact sequence 
\begin{equation}\label{ues}
0\to I_DU_L/I_GU_L\to \mathrm{Ker}\overline{N}/I_GU_L \stackrel{1+\sigma}{\rightarrow} \left(I_GU_L\cdot U_F[p]/I_GU_L\right)^{\Delta}\to 0
\end{equation}
Surjectivity is clear since 
$\left(I_GU_L\cdot U_F[p]/I_GU_L\right)^{\Delta}=N_{\Delta}(I_GU_L\cdot U_F[p]/I_GU_L)$ and not only $U_F[p]\subseteq U_L[N_D]$ but actually $U_F[p]\subseteq \mathrm{Ker}\overline{N}$. To describe the kernel, set $Y=\{u\in U_L\,|\,u^{1+\sigma}\in I_GU_L\}$ and note that $I_GU_L\subseteq Y$ and $Y/I_G U_L$ equals the kernel of (\ref{ues}). 
Now if $u\in T$, then there exists $v\in U_L$ such that
$$u^{1+\sigma}=v^{1-\tau}$$
and in particular 
$$u^2=v^{1-\tau}u^{1-\sigma}\in I_DU_L$$
Then $u\in I_DU_L$ and since clearly $I_DU_L\subseteq Y$, we have in fact the equality.
%Then consider the following exact sequence
%$$0\to I_{\langle\sigma\rangle}U_L\to A \stackrel{1+\sigma}{\longrightarrow} I_G U_L\cap U_K\to 0$$ 
%This gives rise to the exact sequence
%$$0\to I_DU_L/I_G U_L\to A/I_GU_L \stackrel{1+\sigma}{\longrightarrow} I_G U_L\cap U_K/N_{\langle\sigma\rangle} (I_G U_L)\to 0$$
%since $I_{\langle\sigma\rangle}U_L\cdot I_G UL=I_DU_L$. Now $N_{\langle\sigma\rangle} (I_G U_L)=(I_G U_L)^{\langle\sigma\rangle}=I_G U_L\cap U_K$ and therefore the kernel of (\ref{ues}) is $I_DU_L/I_G U_L$. 
Moreover 
$$\left(I_GU_L\cdot U_F[p]/I_GU_L\right)^{\Delta}\cong\left(U_F[p]/I_GU_L\cap U_F[p]\right)^{\Delta}=$$
$$=U_k[p]/I_GU_L\cap U_k[p]\cong I_GU_L\cdot U_k[p]/I_GU_L$$ 
Then 
$$(U_L[N_G]:U_L[N_G]\cap U_KU_{K'}U_F)=\frac{(U_L[N_G]:I_GU_L\cdot U_F[p])}{(U_L[N_G]\cap U_KU_{K'}U_F:I_GU_L\cdot U_F[p])}=$$ 
$$=\frac{(U_L[N_G]:I_GU_L)}{(I_GU_L\cdot U_F[p]:I_GU_L)\cdot(U_L[N_D]:\mathrm{Ker}\overline{N})}=$$
$$=\frac{|\widehat{H}^{-1}(G,\,U_L)|(\mathrm{Ker}\overline{N}:I_DU_L)}{(I_GU_L\cdot U_F[p]:I_GU_L)\cdot(U_L[N_D]:I_DU_L)}=$$
$$=\frac{|\widehat{H}^{-1}(G,\,U_L)|}{|\widehat{H}^{-1}(D,\,U_L)|}\cdot\frac{(I_GU_L\cdot U_k[p]:I_GU_L)}{(I_GU_L\cdot U_F[p]:I_GU_L)}=$$
$$=\frac{|H^{1}(D,\,U_L)|}{(I_GU_L\cdot U_F[p]:I_GU_L\cdot U_k[p])}$$
by Lemma \ref{nuccio}. Hence the first assertion of the lemma is proved.\\
We now prove the second assertion (this part is actually the same as in the proof of \cite{Lem}, Theorem 2.2). Note that $N_GU_K=N_DU_L=N_GU_{K'}$: in particular $N_GU_K\subseteq U_L^D=U_k$. Therefore
$$(N_GU_L:N_G(U_KU_{K'}U_F))=(N_GU_L:U_F^p\cdot N_GU_K)=$$
$$=\frac{(U_F:U_F^p\cdot N_GU_K)}{(U_F:N_GU_L)}=\frac{(U_F:U_F^p)}{(U_F^pN_GU_K:U_F^p)\cdot(U_F:N_GU_L)}=$$
$$=\frac{(U_F:U_F^p)\cdot(U_F^pU_k:U_F^pN_GU_K)}{(U_F^pU_k:U_F^p)\cdot(U_F:N_GU_L)}$$
Now
$$(U_F^pU_k:U_F^pN_GU_K)=\frac{(U_F^pU_k:U_F^p)}{(U_F^pN_GU_K:U_F^p)}=$$
$$=\frac{(U_k:U_F^p\cap U_k)}{(N_GU_K:U_F^p\cap N_GU_K)}=\frac{(U_k:U_k^p)}{(N_GU_K:U_k^p)}=(U_k:N_GU_K)$$
using (\ref{ppi}) and
$$U_k^p=N_GU_k\subseteq U_F^p\cap N_GU_K=U_F^p\cap N_DU_L\subseteq U_F^p\cap U_k=U_k^p$$
%On the other hand $(U_F^pU_k:U_F^p)=(U_k:U_k\cap U_F^p)=(U_k:U_k^p)$. 
Therefore, using once more (\ref{ppi}),
$$(N_GU_L:N_G(U_KU_{K'}U_F))=\frac{(U_F:U_F^p)(U_k:N_GU_K)}{(U_k:U_k^p)(U_F:N_GU_L)}=$$
$$=\frac{|\widehat{H}^0(D,U_L)|}{|\widehat{H}^0(G,U_L)|}\frac{(U_F:U_F^p)}{(U_k:U_k^p)}=\frac{(U_F:U_F^p)}{(U_k:U_k^p)|H^2(D,\,U_L)|}$$
by Lemma \ref{nuccio}.
\end{proof}

Recall (see Section \ref{intro}) that $\overline{M}$ is our notation for the torsion-free quotient of the $\mathbb{Z}_p$-module $M$. 

\begin{lemma}\label{ultimo}
We have
$$I_GU_L\cap U_F[p]=I_GU_L\cap U_F\quad\textrm{and}\quad I_GU_L\cap U_k[p]=I_GU_L\cap U_k$$ 
Furthermore, there is an isomorphism 
$$I_GU_L\cap U_F\cong \overline{U}_L^G/\overline{U}_F$$
which is $\Delta$-antiequivariant. In particular it induces an isomorphism
$$I_GU_L\cap U_F/I_GU_L\cap U_k\cong \overline{U}_L^D/\overline{U}_F^\Delta=\overline{U}_L^D/\overline{U}_k$$
\end{lemma}
\begin{proof}
The first assertion follows from the fact that $G$ acts trivially on $U_F$ and $U_k$ and therefore
$$I_GU_L\cap U_F\subseteq U_L[N_G]\cap U_F\subseteq U_F[p]\quad \textrm{and}\quad I_GU_L\cap U_k\subseteq U_L[N_G]\cap U_k\subseteq U_k[p]$$
Now consider the map 
$$\phi:I_GU_L\cap U_F \to \overline{U}_L^G/\overline{U}_F$$
defined by $\phi(u^{1-\tau})=\overline{u} \,\,\mathrm{mod}\,\,\overline{U}_F$ ($\overline{u}$ is the class of $u$ in $\overline{U}_L$). First of all, this definition does not depend on the choice of $u$: namely, if $v^{1-\tau}=u^{1-\tau}$, then $vu^{-1}\in U_L^G=U_F$. Of course, the image of $\phi$ is contained in $(\overline{U}_L/\overline{U}_F)^G=\overline{U}_L^G/\overline{U}_F$ (this last equality comes from $H^1(G,\,\overline{U}_F)=\mathrm{Hom}(G,\,\overline{U}_F)=0$ since $\overline{U}_F$ is a free $\mathbb{Z}_p$-module with trivial $G$-action). Moreover $\phi$ is clearly a homomorphism. To see that it is injective, suppose that $\phi(u^{1-\tau})=\overline{1}\,\,\mathrm{mod} \,\,\overline{U}_F$. This means that there exist $\zeta\in \mathrm{tor}(U_L)$ and $v\in U_F$ such that $u=v\zeta$. But since $\mathrm{tor}(U_L)=\mathrm{tor}(U_F)$ by Lemma \ref{torsion}, this implies $u^{1-\tau}=1$. Hence $\phi$ is injective. To prove surjectivity, choose an element $\overline{u}\in \overline{U}_L^{G}$. This means $u^\tau=u\xi$ for some $\xi\in\mathrm{tor}(U_L)$. Then $u^{1-\tau}\in \mathrm{tor}(U_L)=\mathrm{tor}(U_F)\subseteq U_F$ and $\phi(u^{1-\tau})=\overline{u}\,\,\mathrm{mod}\,\,\overline{U}_F$.\\
The map $\phi$ is $\Delta$-antiequivariant, in other words, if $\delta$ generates $\Delta$, we have 
\begin{equation}\label{aeq}
\phi( (u^{1-\tau})^\delta)=\phi(u^{1-\tau})^{-\delta}
\end{equation} 
for any $u^{1-\tau}\in I_GU_L\cap U_F$. In fact, %first of all $u^{1-\tau}=u^{(1-\tau)\tau}=u^{\tau(1-\tau)}$ since $u^{1-\tau}\in U_F$. In particular 
%\begin{equation}\label{resi}
%\phi(u^{1-\tau})=\overline{u}^\tau\,\,\mathrm{mod}\,\,\overline{U}_F 
%\end{equation}
%and
$$(u^{1-\tau})^{\delta}=u^{\sigma(1-\tau)}=u^{(1-\tau^{-1})\sigma}=u^{(1-\tau)(\sum_{i=0}^{p-2}\tau^{i})\sigma}$$
Therefore
$$\phi((u^{1-\tau})^\delta)=\overline{u^{(\sum_{i=0}^{p-2}\tau^{i})\sigma}}\quad \mathrm{mod}\,\,\, \overline{U}_F$$
Hence to verify (\ref{aeq}), we have to check that 
$$v:=u^{(\sum_{i=0}^{p-2}\tau^{i})\sigma+\sigma}\in U_F=U_L^G$$
Now
$$v^{1-\tau^{-1}}=u^{(1-\tau^{-1})(\sum_{i=0}^{p-2}\tau^{i})\sigma+(1-\tau^{-1})\sigma}=(u^{1-\tau})^{\sigma(\sum_{i=0}^{p-2}\tau^{-i})+\sigma}=$$
$$=(u^{1-\tau})^{(p-1)\sigma+\sigma}=(u^{1-\tau})^{p\sigma}=1$$
since $u^{1-\tau}\in U_F\cap I_GU_L$ which means that $\tau$ acts trivially on it and it has order $p$. This proves that $\phi$ is $\Delta$-antiequivariant.
To get the last claim of the proposition note that 
$$(I_GU_L\cap U_F)^\Delta=I_GU_L\cap U_F\cap U_k=I_GU_L\cap U_k$$
and, $2$-divisible modules being $\Delta$-cohomologically trivial, 
$$(\overline{U}_L^G/\overline{U}_F)^\Delta=\overline{U}_L^D/\overline{U}_F^\Delta,\quad \overline{U}_F^\Delta=U_k$$  
Therefore $\phi$ induces an isomorphism
$$I_GU_L\cap U_F/I_GU_L\cap U_k\cong \overline{U}_L^D/\overline{U}_k$$
\end{proof}

We now state and prove a lemma which will allow us to get results for finite sets which are more general than our fixed $T$. 

\begin{lemma}\label{trivialrelations}
Let $S$ be any subset of $T$ which is stable under the action of $D$ and let $S'$ be the union of $S$ with the set of primes above $p$ in $L$. Then, for any subfield $E$ of $L$ containing $k$, 
$$H_{\acute{e}t}^{1}(\mathcal{O}^S_E\textstyle{[\frac{1}{p}]},\,\mathbb{Z}_p(m))\cong H_{\acute{e}t}^{1}(\mathcal{O}^T_E,\,\mathbb{Z}_p(m))\cong H_{\acute{e}t}^{1}(E,\,\mathbb{Z}_p(m))$$
and there is an exact sequence
$$0\to H_{\acute{e}t}^{2}(\mathcal{O}^S_E\textstyle{[\frac{1}{p}]},\,\mathbb{Z}_p(m))\to H_{\acute{e}t}^{2}(\mathcal{O}^T_E,\,\mathbb{Z}_p(m))\to \displaystyle{\bigoplus_{w\in (T\smallsetminus S')_{E}}H_{\acute{e}t}^{1}(k_v,\,\mathbb{Z}_p(m-1))}\to 0$$
where $k_w$ is the residue field of $E$ at $w$. Moreover the function 
\begin{equation}\label{vr}
H\mapsto \prod_{w\in (T\smallsetminus S')_{L^H}}|H_{\acute{e}t}^{1}(k_v,\,\mathbb{Z}_p(m))|
\end{equation} 
which is defined on the set of subgroups of $D$, is trivial on $D$-relations (in the sense of \cite{DD}, Section 2.iii).
\end{lemma}
\begin{proof}
The isomorphism and the exact sequence of the first part of the statement are well-known (see for example \cite{So}, Proposition 1).\\
As for the last part of the lemma, let $H$ be a subgroup of $D$. If $w\in (T \smallsetminus S')_{L^H}$ and $v$ is the prime of $k$ below $w$, we have
\begin{equation}\label{oet}
|\mathrm{tor}H_{\acute{e}t}^{1}(k_w,\,\mathbb{Z}_p(m-1))|=|H_{\acute{e}t}^{1}(k_w,\,\mathbb{Z}_p(m-1))|=|H_{\acute{e}t}^{0}(k_w,\,\mathbb{Q}_p/\mathbb{Z}_p(m-1))|=p^{v_p(\ell_v^{f_w(m-1)}-1)}
\end{equation}
where $\ell_v$ is the rational prime below $v$, $\ell_v^{f_w}=|k_w|$ and $v_p$ is the $p$-adic valuation such that $v_p(p)=1$. We also set $\ell_v^{f_v}=|k_v|$ and $f_{w|v}=f_w/f_v$. Then
$$\prod_{w\in (T\smallsetminus S')_{L^H}}(\ell_v^{f_w(m-1)}-1)=\prod_{v\in T\smallsetminus S'}\prod_{j=1}^{2p}(\ell_v^{f_vj(m-1)}-1)^{\#\{w|v\,\, \textrm{in $L^H$ with $f_{w|v}=j$}\}}$$
Using Theorem 2.36 of \cite{DD} (as explained for instance in Example 2.37 of the same paper), we see that the function
$$H\mapsto \prod_{w\in (T\smallsetminus S')_{L^H}}(\ell_w^{f_w(m-1)}-1)$$
which is defined on the set of subgroups of $D$, is trivial on $D$-relations (being a product of functions which are trivial on $D$-relations). Therefore thanks to (\ref{oet}), we easily see that (\ref{vr}) is trivial on $D$-relations.
\end{proof}

The next proposition can be seen as the $p$-part of Formula \ref{alg}.

\begin{prop}\label{mainp}
Let $p$ be an odd prime and let $L/k$ be a Galois extension of number fields with $\mathrm{Gal}(L/k)=D$. Let $S$ be a finite set of primes of $L$ which is stable under the action of $D$. Then the following formula holds
$$|A^S_{L,\,m}|= p^{-\alpha_m}|A^S_{F,\,m}|\frac{|A^S_{K,\,m}|^2}{|A^S_{k,\,m}|^2}\frac{(U_{L,\,m}:\,U_{K,\,m}U_{K',\,m}U_{F,\,m})}{((\overline{U}_{L,\,m})^{D}:\overline{U}_{k,\,m})}$$   
where $\alpha_m=\mathrm{rk}_{\mathbb{Z}_p}U_{F,\,m}-\mathrm{rk}_{\mathbb{Z}_p}U_{k,\,m}=\mathrm{rk}_{\mathbb{Z}}H_{F,\,m}-\mathrm{rk}_{\mathbb{Z}}H_{k,\,m}$ is as in Formula \ref{alg}. 
\end{prop}
\begin{proof}   
First we prove the proposition in the case where $S=T$. Thanks to Lemma \ref{ultimo}, we have
$$(I_GU_L\cdot U_F[p]:I_GU_L\cdot U_k[p])=\frac{(I_GU_L\cdot U_F[p]:I_GU_L)}{(I_GU_L\cdot U_k[p]:I_GU_L)}=\frac{(U_F[p]:I_GU_L\cap U_F[p])}{(U_k[p]:I_GU_L\cap U_k[p])}=$$
$$=\frac{(U_F[p]:I_GU_L\cap U_F)}{(U_k[p]:I_GU_L\cap U_k)}=\frac{(U_F[p]:U_k[p])}{(I_GU_L\cap U_F:I_GU_L\cap U_k)}=\frac{(U_F[p]:U_k[p])}{(\overline{U}_L^D:\overline{U}_k)}$$
Finally 
$$\frac{(U_F:U_F^p)}{(U_k:U_k^p)}=p^{\alpha_m}(U_F[p]:U_k[p])$$
where $\alpha_m=\mathrm{rk}_{\mathbb{Z}_p}U_{F,\,m}-\mathrm{rk}_{\mathbb{Z}_p}U_{k,\,m}=\mathrm{rk}_{\mathbb{Z}}H^1_{F,\,m}-\mathrm{rk}_{\mathbb{Z}}H^1_{k,\,m}$ (this last equality comes from (\ref{moet})) is as in Formula \ref{alg}.
Now consider the exact sequence of Lemma \ref{es}: we get, using all the preceding lemmas and the remark after Lemma \ref{lem},
$$|A^S_{L,\,m}|=\frac{|A^S_{K,\,m}|^2|H_0(G,\,A_L)||H^1(D,\,U_L)|}{|A^S_{k,\,m}||H_0(G,\,A_L)^{\Delta}||H^2(D,\,U_L)|}=$$
$$=|A^S_{F,\,m}|\frac{|A^S_{K,\,m}|^2}{|A^S_{k,\,m}|^2}\frac{(U_L:\,U_FU_KU_{K'})(I_GU_L\cdot U_F[p]:I_GU_L\cdot U_k[p])(U_k:U_k^p)}{(U_F:U_F^p)}=$$
$$=p^{-\alpha_m}|A^S_{F,\,m}|\frac{|A^S_{K,\,m}|^2}{|A^S_{k,\,m}|^2}\frac{(U_L:\,U_FU_KU_{K'})}{(\overline{U}_L^D:\overline{U}_k)}$$
To get the general statement, note that, by Lemma \ref{trivialrelations}, for any subgroup $H$ of $D$, the function
$$H\mapsto \frac{|A_{L^H,\,m}^T|}{|A_{L^H,\,m}^S|}$$
is trivial on the relation (\ref{relation}). 
\end{proof}

We now deal with the general proof of Formula \ref{alg}. We are going to use the language and some results of the theory of cohomological Mackey functors: instead of recalling definitions we prefer to directly refer the reader to \cite{Bo}, Section 1. The next result is essentially a consequence of the fact that $D=D_p$ is not $\ell$-hypoelementary (a group is $\ell$-hypoelementary if it has a normal $\ell$-subgroup with cyclic quotient), provided that $\ell$ is any prime different from $p$.

\begin{prop}\label{nonp}
Let $\ell$ be a rational prime different from $p$. Let $S$ be a finite set of primes of $L$ which is stable under the action of $\mathrm{Gal}(L/k)$. Then there is an isomorphism of abelian groups
$$H_{\acute{e}t}^{2}(\mathcal{O}_L^S\textstyle{[\frac{1}{\ell}]},\,\mathbb{Z}_\ell(m))\oplus H_{\acute{e}t}^{2}(\mathcal{O}_k^S\textstyle{[\frac{1}{\ell}]},\,\mathbb{Z}_\ell(m))^2\cong H_{\acute{e}t}^{2}(\mathcal{O}_F^S\textstyle{[\frac{1}{\ell}]},\,\mathbb{Z}_\ell(m))\oplus H_{\acute{e}t}^{2}(\mathcal{O}_K^S\textstyle{[\frac{1}{\ell}]},\,\mathbb{Z}_\ell(m))^2$$
\end{prop}
\begin{proof}
Note that the function which assigns to any subgroup $H$ of $D$ the abelian group $H_{\acute{e}t}^{2}(\mathcal{O}_{L^H}^S\textstyle{[\frac{1}{\ell}]},\,\mathbb{Z}_\ell(m))$ is a cohomological Mackey functor on $D$. Since $\ell\ne p$, $D=D_p$ is not $\ell$-hypoelementary which allow us to apply Theorem 1.8 of \cite{Bo} to conclude. 
\end{proof}

Together with Formula \ref{ind} and a generalization of a result of Brauer (see \cite{Ba}, Theorem 5.1), the fact that $D=D_p$ is not $\ell$-hypoelementary if $\ell\ne p$ can also be used to give a proof of the next lemma. Here we give another proof to show that one can prove Formula \ref{alg} without using Formula \ref{ind}.

\begin{lemma}\label{indp}
Let $S$ be a finite set of primes of $L$ which is stable under the action of $D=\mathrm{Gal}(L/k)$. Then the number $u_m$ is a power of $p$: more precisely the following equality holds
\begin{equation}\label{ugu}
\frac{(U_{L,\,m}:\,U_{K,\,m}U_{K',\,m}U_{F,\,m})}{((\overline{U}_{L,\,m})^{D}:\overline{U}_{k,\,m})}=\frac{(H_{L,\,m}:\,H_{F,\,m}H_{K,\,m}H_{K',\,m})((\overline{H}_{F,\,m})^{\Delta}:\overline{H}_{k,\,m})}{((\overline{H}_{L,\,m})^{D}:\overline{H}_{k,\,m})}=u_m
\end{equation}
where $H_{E,\,m}=H^1(\mathcal{O}_E^S,\,\mathbb{Z}(m))$ for any subfield $E$ of $L$ containing $k$.
\end{lemma}
\begin{proof}
Note that, thanks to (\ref{moet}), the $p$-part of the index on the right-hand side is precisely the index on the left-hand side. Then we are left to show that, for any fixed prime $\ell\ne p$, 
\begin{equation}\label{vvv}
\frac{(V_{L}:\,V_{K}V_{K'}V_{F})((\overline{V}_{F})^{D}:\overline{V}_{k})}{((\overline{V}_{L})^{D}:\overline{V}_{k})}=1
\end{equation} 
where 
$$V_E=H^1_{\acute{e}t}(\mathcal{O}_E^S\textstyle{[\frac{1}{\ell}]},\,\mathbb{Z}_\ell(m))$$ 
But this is easy, using Lemma \ref{lem}. Details are as follows: first of all 
\begin{equation}\label{trivella}
(V_L:V_KV_{K'}V_F)=(N_GV_L:N_G(V_KV_{K'}V_F))\cdot(V_L[N_G]:V_L[N_G]\cap V_KV_{K'}V_F)=1
\end{equation} 
since
\begin{itemize}
	\item $N_G(V_L)=V_L^G=V_F$ (because $V_L$ and $\mathbb{Z}_\ell(m)$ are $G$-cohomologically trivial);
	\item $N_G(V_F)=V_F^p=V_F$ (because $G$ acts trivially on $V_F$ and $V_F$ is $p$-divisible);
  \item $V_L[N_G]=I_GV_L\subseteq V_KV_{K'}$ (because $V_L$ is $G$-cohomologically trivial and an appropriate version of Lemma \ref{esto} holds).
\end{itemize}
Moreover $\overline{V}_L^G=\overline{V}_F$ (again because $\overline{V}_L$ and $\mathbb{Z}_\ell(m)$ are $G$-cohomologically trivial), showing that (\ref{vvv}) holds.
\end{proof}

The proof of Formula \ref{alg} is then achieved, because (\ref{moet}) allows us to glue together Proposition \ref{mainp}, Proposition \ref{nonp} and Lemma \ref{indp}. Note that, if $H_{F,\,m}$ has no $p$-torsion, then
$u_m=(H_{L,\,m}:H_{K,\,m}H_{K',\,m}H_{F,\,m})$ (this follows easily from the last assertion of Lemma \ref{ultimo}).
  
\section{Computations with regulators}\label{regulcomp}

In this section we are going to prove Formula \ref{ind}, translating the higher units index 
$$u_m=\frac{(H_{L,\,m}:\,H_{F,\,m}H_{K,\,m}H_{K',\,m})((\overline{H}_{F,\,m})^{\Delta}:\overline{H}_{k,\,m})}{((\overline{H}_{L,\,m})^D:\overline{H}_{k,\,m})}$$ 
of Formula \ref{alg} in terms of motivic (or higher) regulators, whose definition and basic properties we briefly recall (we refer the reader to \cite{Neu}, §1). Recall from Section \ref{intro} that, if $E$ is any number fields, we have set $H_{E,\,m}=H^1(\mathcal{O}^S_E,\,\mathbb{Z}(m))$.\\ 
Let $m\geq 2$ be a natural number. If $X(E)=\mathrm{Hom}(E,\,\mathbb{C})$ is the set of complex embeddings of $E$, the $m$-th regulator map we shall consider is a homomorphism
$$\rho_{E,\,m}:H_{E,\,m}\longrightarrow \bigoplus_{\beta\in X(E)}(2\pi i)^{m-1}\mathbb{R}$$
This map is obtained by composing the natural map $H_{E,\,m}\to H_{E,\,m}\otimes \mathbb{Q}$ with the Beilinson regulator map $K_{2m-1}(\mathcal{O}_E)\otimes \mathbb{Q}\to \bigoplus_{\beta\in X(E)}(2\pi i)^{m-1}\mathbb{R}$ via the isomorphism (\ref{mok}). Moreover, the image of $\rho$ is contained in the subgroup of $\bigoplus_{\beta\in X(E)}(2\pi i)^{m-1}\mathbb{R}$ which is fixed by complex conjugation: 
$$\rho_{E,\,m}:H_{E,\,m}\longrightarrow \left(\bigoplus_{\beta\in X(E)}(2\pi i)^{m-1}\mathbb{R}\right)^+$$
The kernel of $\rho$ is exactly the torsion subgroup of $H_{E,\,m}$ and, thanks to Borel's theorem (and the fact that Beilinson's regulator map is twice Borel's regulator map), we know that $\rho$ induces an isomorphism $H_{E,\,m}\otimes\mathbb{R}\cong(\bigoplus_{\beta\in X(E)}(2\pi i)^{m-1}\mathbb{R})^+$. Then the $m$-th regulator of $E$, denoted $R_{E,\,m}$, is the covolume of the lattice $\rho_{E,\,m}(H_{E,\,m})$ as a subset of the real vector space $(\bigoplus_{\beta\in X(E)}(2\pi i)^{m-1}\mathbb{R})^+$. Finally, thanks to the functorial properties of the Beilinson regulator, if a Galois action is defined on $E$, then $\rho$ is invariant under this action. Then we get a generalization of Dirichlet's theorem on units (compare with \cite{NSW}, Proposition 8.6.11). First we introduce some notation: if $\Gamma$ is a finite group and $\Gamma'$ is a subgroup of $\Gamma$, for a $\mathbb{Z}[\Gamma']$-module $M$, we denote by $\mathrm{Ind}_{\Gamma'}^{\Gamma}M$ the $\mathbb{Z}[\Gamma]$-module induced by $M$. Moreover $S_m(E)$ denotes the set of archimedean places of $E$ (resp. non-real archimedean places of $E$) if $m$ is odd (resp. if $m$ is even) and, if $E'$ is a subfield of $E$, $R(E/E')$ denotes the set of real primes of $E'$ which becomes complex in $E$. Finally, for a place $\mathfrak{p}$ in $S_m(k)$, we denote by $D_\mathfrak{p}$ any of the decomposition groups of places above $\mathfrak{p}$ in $L$. We will use the notation $\mathbbm{1}_{D_\mathfrak{p}}$ (resp. $\varepsilon_{D_\mathfrak{p}}$) for the trivial $\mathbb{Z}[D_{\mathfrak{p}}]$-module (resp. for the $\mathbb{Z}[D_{\mathfrak{p}}]$ given by the sign). 

\begin{teo}\label{hdt}[Higher Dirichlet's Theorem]
Let $E/k$ be a Galois extension of number fields with Galois group $\Gamma$. If $m\geq 2$ is odd, then there is an isomorphism of $\mathbb{Q}[\Gamma]$-modules
$$H^1(\mathcal{O}_E,\,\mathbb{Z}(m))\otimes_{\mathbb{Z}}\mathbb{Q}\cong\left(
\bigoplus_{\mathfrak{p}\in S_m(k)}\mathrm{Ind}_{D_{\mathfrak{p}}}^\Gamma \mathbbm{1}_{D_{\mathfrak{p}}}\right)\otimes_{\mathbb{Z}} \mathbb{Q}$$
while, if $m\geq 2$ is even, there is an isomorphism of $\mathbb{Q}[\Gamma]$-modules
$$H^1(\mathcal{O}_E,\,\mathbb{Z}(m))\otimes_{\mathbb{Z}}\mathbb{Q}\cong
\left(\bigoplus_{\mathfrak{p}\in R(E/k)}\mathrm{Ind}_{D_{\mathfrak{p}}}^\Gamma \varepsilon_{D_{\mathfrak{p}}}\oplus
\bigoplus_{\mathfrak{p}\in S_{m+1}(k)}\mathrm{Ind}_{D_{\mathfrak{p}}}^\Gamma\mathbbm{1}_{D_{\mathfrak{p}}}\right)\otimes_{\mathbb{Z}} \mathbb{Q}$$
\end{teo}
\begin{proof}
As in the classical case, it is easy to see that 
$$\left(\bigoplus_{\beta\in X(E)}(2\pi i)^{m-1}\mathbb{R}\right)^+\cong\left\{\begin{array}{ll}\left(\bigoplus_{\mathfrak{p}\in S_m(k)}\mathrm{Ind}_{D_{\mathfrak{p}}}^\Gamma \mathbbm{1}_{D_{\mathfrak{p}}}\right)\otimes_{\mathbb{Z}} \mathbb{R}&\textrm{if $m$ is odd}\\\left(\bigoplus_{\mathfrak{p}\in R(E/k)}\mathrm{Ind}_{D_{\mathfrak{p}}}^\Gamma \varepsilon_{D_{\mathfrak{p}}}\oplus
\bigoplus_{\mathfrak{p}\in S_{m+1}(k)}\mathrm{Ind}_{D_{\mathfrak{p}}}^\Gamma\mathbbm{1}_{D_{\mathfrak{p}}}\right)\otimes_{\mathbb{Z}} \mathbb{R}&\textrm{if $m$ is even}\end{array}\right.$$ 
as $\mathbb{R}[\Gamma]$-modules (here $S_{m+1}(k)$ is just a notation trick for the set of complex places of $k$ when $m$ is odd). The proof of the theorem then follows by the fact that $\rho_{E,\,m}$ is a $\mathbb{R}[\Gamma]$-isomorphism as explained above, together with Lemma 8.6.10 in \cite{NSW}. 
\end{proof}

To translate $u_m$ in terms of regulators, we are going to use the technique of regulators costants, introduced by the Dokchitser brothers. We first define a scalar product on $H_{E,\,m}$ in the following way: denote by $\langle-,\,-\rangle_\infty$ the standard scalar product on $\mathbb{C}^{|X(E)|}$. Then, for $u,\,v\in H_{E,\,m}$ we set
$$\langle u,\,v\rangle_{E,\,m}=\langle \rho_{E,\,m}(u),\,\rho_{E,\,m}(v)\rangle_\infty$$
It is immediate to see that $\langle -,\,-\rangle_{E,\,m}$ is a $\mathbb{Z}$-bilinear map on $H_{E,\,m}\times H_{E,\,m}$ which takes values in $\mathbb{R}$. Moreover if $E/E'$ is a finite extension and $u,\,v\in H_{E',\,m}$, then 
\begin{equation}\label{pinardelasrozas}
\langle u,\,v\rangle_{E,\,m}=[E:\,E']\langle u,\,v\rangle_{E',\,m}
\end{equation}
Furthermore, if $E/E'$ is Galois, then $\langle -,\,-\rangle_{E,\,m}$ is invariant with respect to the Galois action.
The regulator map being trivial on $\mathrm{tor}_{\mathbb{Z}}H_{E,\,m}$, $\langle -,\,-\rangle_{E,\,m}$ defines a $\mathbb{Z}$-bilinear map on $\overline{H}_{E,\,m}$.\\

Next we recall the definition of the regulator (or Dokchitser) constant in our particular case (see \cite{DD}, Definition 2.13 and Remark 2.27 or \cite{Ba}, Definition 2.5).

\begin{defi}
Let $M$ be a $\mathbb{Z}[D]$-module which is $\mathbb{Z}$-free of finite rank and such that $M\otimes \mathbb{Q}$ is self-dual. Let $\langle\cdot,\,\cdot\rangle$ be a $D$-invariant non-degenerate $\mathbb{Z}$-bilinear pairing with values in $\mathbb{R}$. Then the regulator constant of $M$ is  
$$\mathcal{C}(M)=\frac{\mathrm{det}(\langle-,\,-\rangle)\,\mathrm{det}\left(\frac{1}{|D|}\langle-,\,-\rangle_{|_{M^{D}}}\right)^{2}}{\mathrm{det}\left(\frac{1}{|G|}\langle\cdot,\,\cdot\rangle_{|_{M^G}}\right)\mathrm{det}\left(\frac{1}{|\langle\sigma\rangle|}\langle-,\,-\rangle_{|_{M^{\langle\sigma\rangle}}}\right)^{2}}\in \mathbb{Q}^\times$$
\end{defi}

We will be interested in the case where $M=\overline{H}_{L,\,m}$ and we will use $\langle-,\,-\rangle_{L,\,m}$ to compute its regulator constant. Proposition \ref{bar} below shows that $\langle-,\,-\rangle_{L,\,m}$ is non-degenerate, $R_{L,\,m}$ being non-zero if $\overline{H}_{L,\,m}\ne0$ (note also that the fact that $\overline{H}_{L,\,m}\otimes \mathbb{Q}$ is self-dual is equivalent to the existence of a $D$-invariant non-degenerate $\mathbb{Z}$-bilinear pairing on $\overline{H}_{L,\,m}$).\\

For any subfield $E$ of $L$ containing $k$, set 
$$\lambda_{E,\,m}=\left|\mathrm{ker}\left(\overline{H}_{E,\,m}\hookrightarrow\overline{H}_{L,\,m}^{\mathrm{Gal}(L/E)}\right)\right|$$
It is immediate to see that
\begin{equation}\label{pinardemajadahonda}
\lambda_{E,\,m}=\left|\mathrm{coker}\left(H^1(\mathrm{Gal}(L/E),\,\mathrm{tor}_{\mathbb{Z}}H_{E,\,m})\to H^1(\mathrm{Gal}(L/E),H_{E,\,m})\right)\right|
\end{equation}
and sometimes this description will be useful: for example it shows immediatly that $\lambda_{E,\,m}$ is well defined (i.e. the order of the above kernel is indeed finite).\\ 

The following result is similar to Lemma 2.12 and Proposition 2.15 in \cite{Ba}: we sketch the proof since there are slight differences from the proof given in that paper. Recall that $r_1(E)$ (resp. $r_2(E)$) is the number of real places (resp. complex places) of the number field $E$.

\begin{lemma}\label{bar}
We have
$$\det(\langle \cdot,\,\cdot\rangle_{E,\,m})=(-1)^{m-1}2^{r_2(E)}(R_{E,\,m})^2$$
Moreover
$$\mathcal{C}(\overline{H}_{L,\,m})=\left(\frac{R_{L,\,m}R_{k,\,m}^2}{R_{F,\,m}R_{K,\,m}^2}\frac{\lambda_{F,\,m}\lambda_{K,\,m}^2}{\lambda_{L,\,m}\lambda_{k,\,m}^2}\right)^2$$
\end{lemma} 
\begin{proof}
The determinant of 
$$\langle-,\,-\rangle_\infty:(\bigoplus_{\beta\in X(E)}(2\pi i)^{m-1}\mathbb{R})^+\times (\bigoplus_{\beta\in X(E)}(2\pi i)^{m-1}\mathbb{R})^+\to \mathbb{R}$$
is easily seen to be $(-1)^{m-1}2^{r_2(E)}$. Then the first assertion follows by well-know properties of scalar products and the definition of $R_{E,\,m}$.\\
The proof of the second claim follows from the fact that, for any subgroup $H$ of $D$, 
$$\lambda_{L^H,\,m}\det({\langle -,\,-\rangle_{L,\,m}}_{|_{\overline{H}_{L^H,\,m}}})=\det(\langle -,\,-\rangle_{L^H,\,m})$$
(this follows easily from (\ref{pinardemajadahonda}) and (\ref{pinardelasrozas}), see also \cite{Ba}, Lemma 2.13 and Lemma 2.14) and the fact that the function 
$$H\mapsto (-1)^{m-1}2^{r_2(L^H)}$$
defined of the set of subgroups $H$ of $D$ is trivial on relations (see \cite{DD}, Example 2.37).
\end{proof}

The preceding lemma shows that the regulator constant of $\overline{H}_{L,\,m}$ is related to regulators. The next one shows that it is also related to a higher unit index. First we need some notation: we denote by $\mathbbm{1}$ the trivial $\mathbb{Q}[D]$-module, by $\varepsilon$ the $\mathbb{Q}[D]$-module given by the sign and by $\omega$ an irreducible $(p-1)$-dimensional $\mathbb{Q}[D]$-module (which is unique up to isomorphism). It can be checked that $K_0(\mathbb{Q}[D])$ is a torsion-free $\mathbb{Z}$-module of rank $3$ which is spanned by $\mathbbm{1}$, $\varepsilon$ and $\omega$. We denote by $\langle\cdot,\,\cdot\rangle$ denotes the symmetric scalar product on $K_0(\mathbb{Q}[D])$ for which the basis $\{\mathbbm{1},\,\varepsilon,\,\omega\}$ is orthonormal.

\begin{lemma}\label{rep}[Bartel]
With the notation above, we have
$$\mathcal{C}(\overline{H}_{L,\,m})^{-\frac{1}{2}}=p^{\frac{\langle H_{L,\,m}\otimes\mathbb{Q},\,\mathbbm{1}-\varepsilon-\omega\rangle}{2}}(\overline{H}_{L,\,m}:\overline{H}_{L,\,m}^{\langle\sigma\rangle}\overline{H}_{L,\,m}^{^{\langle\omega^2\sigma\rangle}}\overline{H}_{L,\,m}^{G})$$
\end{lemma} 
\begin{proof}
See \cite{Ba}, Lemma 4.8.
\end{proof}

\begin{lemma}\label{reg}
The following equality holds
$$\frac{(R_{K,\,m})^2R_{F,\,m}}{(R_{k,\,m})^2R_{L,\,m}}=p^{-\alpha_m}(\overline{H}_{L,\,m}:\overline{H}_{L,\,m}^{\langle\sigma\rangle}\overline{H}_{L,\,m}^{^{\langle\tau^2\sigma\rangle}}\overline{H}_{L,\,m}^{G})\frac{\lambda_{F,\,m}\lambda_{K,\,m}^2}{\lambda_{k,\,m}^2}$$
where $\alpha_m=\mathrm{rk}_{\mathbb{Z}}H_{F,\,m}-\mathrm{rk}_{\mathbb{Z}}H_{k,\,m}$ is the same as in Formula \ref{alg}.
\end{lemma}
\begin{proof}
Thanks to Lemma \ref{bar} and Lemma \ref{rep}, we only have to show that 
$$\langle H_{L,\,m}\otimes\mathbb{Q},\,\mathbbm{1}\rangle-\langle H_{L,\,m}\otimes\mathbb{Q},\,\varepsilon\rangle-\langle H_{L,\,m}\otimes\mathbb{Q},\,\omega\rangle=-2\alpha_m$$
Suppose that $\mathfrak{P}$ is an infinite prime in $L$, such that its decomposition group $D_{\mathfrak{P}}$ in $L/k$ has order $2$. Then it is not difficult to check that
$$\mathrm{Ind}_{D_{\mathfrak{P}}}^{D}1_{D_{\mathfrak{P}}}= 1+ \omega\quad\textrm{and}\quad \mathrm{Ind}_{D_{\mathfrak{P}}}^{D}\varepsilon_{D_{\mathfrak{P}}}= \varepsilon+ \omega$$
as $\mathbb{Q}[D]$-modules. Furthermore
$$\mathrm{Ind}_{\{1\}}^{D}1_{\{1\}}= 1+ \varepsilon + 2\omega$$
the left-hand side being isomorphic to the $\mathbb{Q}[D]$-module corresponding to the regular representation.\\
Denote by $r(F/k)$ the cardinality of $R(F/k)$. If $m$ is odd, by Theorem \ref{hdt}, we have
$$H_{L,\,m}\otimes_{\mathbb{Z}}\mathbb{Q}= (r_1(k)+r_2(k)-r(F/k))(1+ \varepsilon + 2\omega) + r(F/k)(1+\omega)$$
Therefore 
$$\langle H_{L,\,m}\otimes_{\mathbb{Z}}\mathbb{Q},\,1-\varepsilon-\omega\rangle%=r_1(k)+r_2(k)-(r_1(k)+r_2(k)-r_\infty(F/k))-(2r_1(k)+2r_2(k)-r_\infty(F/k))
=2(r(F/k)-r_1(k)-r_2(k))=-2(\mathrm{rk}_{\mathbb{Z}}H_{F,\,m}-\mathrm{rk}_{\mathbb{Z}}H_{k,\,m})$$
On the other hand, if $m$ is even, again by Theorem \ref{hdt}, we have
$$H_{L,\,m}\otimes_{\mathbb{Z}}\mathbb{Q}= r(F/k)(\varepsilon+\omega)+r_2(k)(1+ \varepsilon + 2\omega)$$
Therefore 
$$\langle H_{L,\,m}\otimes_{\mathbb{Z}}\mathbb{Q},\,1-\varepsilon-\omega\rangle%=r_\infty(F/k)+r_2(k)-(r_\infty(F/k)+r_2(k))-2(r_\infty(F/k)+r_2(k))
=-2(r(F/k)+r_2(k))=-2(\mathrm{rk}_{\mathbb{Z}}H_{F,\,m}-\mathrm{rk}_{\mathbb{Z}}H_{k,\,m})$$ 
and this concludes the proof.
\end{proof}

In order to get the proof of Formula \ref{ind}, we need to compare the right hand side of the formula of Lemma \ref{reg} with $u_m$. For this we need some result about $\lambda_{E,\,m}$.

\begin{lemma}\label{restr}
Let $Q$ be any of the subgroups of order $2$ in $D$ and let $M$ be a finitely generated $\mathbb{Z}_2$-module on which $D$ acts. For any $j\geq 1$, the cohomological restriction induces isomorphisms
$$H^j(D,\,M)\cong H^j(Q,\,M)$$
\end{lemma}
\begin{proof}
For the first assertion, note that $H^j(D,\,M)$ is an abelian $2$-group since $M$ is a finitely generated $\mathbb{Z}_2$-module. Therefore it is well-known that the restriction map
\begin{equation}\label{cicci}
H^j(D,\,M)\longrightarrow H^j(Q,\,M)
\end{equation}
is injective since $Q$ is a $2$-Sylow subgroup of $D$ (see \cite{CE}, Theorem 10.1 in Chapter XII). However by the description of the image of the restriction given in \cite{CE}, Theorem 10.1 in Chapter XII, we see that the map in (\ref{cicci}) is in fact bijective since $Q$ has trivial intersection with any of its conjugates different from $Q$ itself (in the language of \cite{CE}, any element of $H^j(Q,\,M)$ is \emph{stable}).
\end{proof}

\begin{lemma}\label{maria}
The following equality holds 
$$(\overline{H}_{L,\,m}:\overline{H}_{L,\,m}^{\langle\sigma\rangle}\overline{H}_{L,\,m}^{^{\langle\tau^2\sigma\rangle}}\overline{H}_{L,\,m}^{G})\frac{\lambda_{F,\,m}\lambda_{K,\,m}^2}{\lambda_{k,\,m}^2}=\frac{(H_{L,\,m}:H_{K,\,m}H_{K',\,m}H_{F,\,m})\lambda_{K,\,m}}{\lambda_{k,\,m}}$$
\end{lemma}
\begin{proof}
Using the notation introduced in Section \ref{secp}, we have, thanks to (\ref{trivella}) and Lemma \ref{torsion},
$$(H_{L,\,m}:H_{K,\,m}H_{K',\,m}H_{F,\,m})=(U_{L,\,m}:U_{K,\,m}U_{K',\,m}U_{F,\,m})=(\overline{U}_{L,\,m}:\overline{U}_{K,\,m}\overline{U}_{K',\,m}\overline{U}_{F,\,m})$$
With similar arguments, it is easy to prove that we also have
$$(\overline{H}_{L,\,m}:\overline{H}_{L,\,m}^{\langle\sigma\rangle}\overline{H}_{L,\,m}^{^{\langle\tau^2\sigma\rangle}}\overline{H}_{L,\,m}^{G})=(\overline{U}_{L,\,m}:\overline{U}_{K,\,m}\overline{U}_{K',\,m}\overline{U}_{L,\,m}^{G})$$
Furthermore, $\lambda_{K,\,m}$ is a power of $2$ and if, for a prime $\ell$, $(\lambda_{k,\,m})_\ell$ denotes the exact power of $\ell$ dividing $\lambda_{k,\,m}$, then we have
$$(\lambda_{k,\,m})_2=\lambda_{K,\,m}$$ 
(use (\ref{pinardemajadahonda}) and Lemma \ref{restr} applied to $M=\mathrm{tor}_\mathbb{Z}H_{L,\,m}\otimes \mathbb{Z}_2$ and $M=H_{L,\,m}\otimes \mathbb{Z}_2$). Moreover $\lambda_{F,\,m}=1$ or $p$ (since $\mathrm{tor}(U_{F,\,m})$ is cyclic) and $(\lambda_{k,\,m})_p\leq \lambda_{F,\,m}$ (this is immediate, for example use (\ref{pinardemajadahonda})). Thus we are left to check that
\begin{equation}\label{cappero}
(\overline{U}_{K,\,m}\overline{U}_{K',\,m}\overline{U}_{L,\,m}^{G}:\overline{U}_{K,\,m}\overline{U}_{K',\,m}\overline{U}_{F,\,m})=\frac{\lambda_{F,\,m}}{(\lambda_{k,\,m})_p}
\end{equation}
Note that 
$$(\overline{U}_{K,\,m}\overline{U}_{K',\,m}\overline{U}_{L,\,m}^{G}:\overline{U}_{K,\,m}\overline{U}_{K',\,m}\overline{U}_{F,\,m})=(\overline{U}_{L,\,m}^{G}:(\overline{U}_{K,\,m}\overline{U}_{K',\,m}\overline{U}_{F,\,m})^{G})$$
In the rest of the proof we let the subscript $m$ drop and we shall suppose that $\overline{U}_{L}^{G}\neq \overline{U}_F$ (otherwise (\ref{cappero}) is trivially true).\\
Suppose first that $I_GU_L\cap U_k\ne I_GU_L\cap U_F$ (which is equivalent to $\lambda_{F}=(\lambda_{k})_p=p$ by Lemma \ref{ultimo}). Let $u\in U_L$ be such that the class of $[u]\in \overline{U}_L$ is nontrivial in $\overline{U}_L^G/\overline{U}_F$: then $u^{\tau-1}=\zeta$ generates $I_GU_L\cap U_F$ which is cyclic of order $p$ (note that we also have $u^{\tau^{-1}-1}=\zeta^{-1}$). Then $u^{(1+\sigma)}\in U_K\subseteq U_{K}U_{K'}U_{F}$ and
$$(u^{(1+\sigma)})^{(\tau-1)}=u^{(\tau-1)(1+\sigma)}=u^{\tau-1}u^{(\tau-1)\sigma}=\zeta^{1-\sigma}=\zeta^2$$
because $\sigma$ acts as $-1$ on $I_GU_L\cap U_F$ (the latter being different from $I_GU_L\cap U_k$ by assumption).  
This shows at once that $[u]^{1+\sigma}\in(\overline{U}_{K}\overline{U}_{K'}\overline{U}_{F})^{G}$ but $[u]^{1+\sigma}\notin \overline{U}_F$ (since $\sigma$ acts trivially on $\overline{U}_L^G/\overline{U}_F$ and $[u]\notin \overline{U}_F$) and therefore $\overline{U}_{L}^{G}=(\overline{U}_{K}\overline{U}_{K'}\overline{U}_{F})^{G}$ (since $\overline{U}_{L}^{G}/\overline{U}_{F}$ has order $p$).\\ 
Now suppose that $I_GU_L\cap U_k= I_GU_L\cap U_F$ and we have to show that $(\overline{U}_{K}\overline{U}_{K'}\overline{U}_{F})^{G}=\overline{U}_F$. First of all we observe that 
\begin{equation}\label{yolanda}
U_K^p\cap U_k=U_k^p
\end{equation} 
In fact, of course $U_K^p\cap U_k\supseteq U_k^p$. Conversely, if $u\in U_K$ and $u^p\in U_k$, then $u^{p(\tau-1)}=1$. In particular $u^{\tau-1}\in I_GU_L\cap U_L[p]= I_GU_L\cap U_F=I_GU_L\cap U_k$ (use Lemma \ref{torsion} and Lemma \ref{ultimo}). But this means that
$$u^{\sigma(\tau-1)}=u^{\tau-1}$$
and since $u^{\sigma(\tau-1)}=u^{\tau^{-1}-1}$ (because $u^{\tau^{-1}-1}=u^{-\tau+1}\in U_k$), this implies that $u^{\tau^2}=u$. Therefore $u\in U_F\cap U_K=U_k$ and therefore $U_K^p\cap U_k\subseteq U_k^p$. Now consider an element in $(\overline{U}_{K}\overline{U}_{K'}\overline{U}_{F})^G$: it can be written as $[uv^{\tau}w]$ with $u,\,v\in U_K$ and $w\in U_F$ satisfying
$$(uv^\tau w)^{\tau-1}\in \mathrm{tor}U_L$$
Actually $(uv^\tau w)^{\tau-1}=(uv^\tau)^{\tau-1}=\xi\in U_k[p]$ (thanks to $I_GU_L\cap U_k= I_GU_L\cap U_F$, Lemma \ref{torsion} and Lemma \ref{ultimo}). This implies that
$$\frac{u^{p\tau}}{v^{p\tau}}=\frac{u^p}{v^{p\tau^2}}\in U_{K'}$$
In particular
$$\left(\frac{u^p}{v^{p\tau^2}}\right)^{\tau^2\sigma}=\frac{u^p}{v^{p\tau^2}}$$
A quick calculation then shows that $(uv)^p=(uv)^{p\tau^2}$. This means that $(uv)^p\in U_F\cap U_K=U_k$. Then 
$(uv)^p\in U_K^p\cap U_k=U_k^p$ thanks to (\ref{yolanda}). Then $u=v^{-1}z$ with $z\in U_k$ and therefore
$$(uv^\tau w)^{\tau-1}=(uv^\tau)^{\tau-1}=\left(\frac{v^\tau}{v}\right)^{\tau-1}=\xi \in U_k[p]$$
This implies that 
$$\frac{v^{p\tau}}{v^p}\in U_F\cap I_GU_L=U_F[p]\cap I_GU_L$$
which means $v^{\tau}/v\in \mathrm{tor}_{\mathbb{Z}_p}U_L\subseteq U_F$. Then $\xi=1$ and $[uv^\tau w]\in \overline{U}_F$. This concludes the proof.
\end{proof}

Collecting all the results we have proved so far, we get the proof of Formula \ref{ind} and therefore of the higher Brauer-Kuroda relation (\ref{bk}).

\end{document}